\documentclass[11pt]{article}

\usepackage[a4paper,margin=1in]{geometry}
\usepackage{amsmath,amssymb,amsthm,mathtools}
\usepackage{enumitem}
\usepackage{microtype}
\usepackage{hyperref}
\usepackage{tikz}
\usepackage{tikz-cd}

\hypersetup{
    colorlinks=true,
    linkcolor=blue,
    citecolor=blue,
    urlcolor=blue
}

\newtheorem{theorem}{Theorem}[section]
\newtheorem{proposition}[theorem]{Proposition}
\newtheorem{lemma}[theorem]{Lemma}
\newtheorem{corollary}[theorem]{Corollary}

\theoremstyle{definition}

\theoremstyle{remark}
\newtheorem{remark}[theorem]{Remark}

\newcommand{\Dur}{\operatorname{Dur}}
\newcommand{\cH}{\mathcal{H}}
\newcommand{\Z}{\mathbb{Z}}

\title{\textbf{The range and omitted values of a certain sequence involving the partition function
}}

\author{
Kalyan Chakraborty\thanks{
Department of Mathematics,
SRM University-AP,
Neerukonda, Mangalagiri Mandal,
Guntur District,
Andhra Pradesh 522240, India.
Email: \texttt{kalyan.c@srmap.edu.in}.}
\and
Alisha Kazi\thanks{
Sister Nivedita University,
DG 1/2, New Town, Action Area 1,
Kolkata 700156, India.
Email: \texttt{alishakazimath@gmail.com}.}
\and
Manoj Upreti\thanks{
Department of Mathematics,
SRM University-AP,
Neerukonda, Mangalagiri Mandal,
Guntur District,
Andhra Pradesh 522240, India.
Email: \texttt{manoj.u@srmap.edu.in}.}
}

\date{}

\begin{document}

\maketitle


\begin{abstract}
Let \(p(n)\) denote the ordinary partition function. Motivated by
analogous questions concerning Euler's totient function and its
complementary counting function, we study the range of the
partition-derived sequence \(p(n)-n\). We give combinatorial
interpretations of this sequence and investigate both the attained
and omitted positive integers. We obtain exact and asymptotic
information about the gaps between consecutive attained values and
show that the range is remarkably sparse: its counting function has
order \((\log x)^2\), and consequently the range has natural density
zero. We also extend the discussion to partitions whose Durfee
square has side at least a fixed positive integer.
\end{abstract}

\medskip


\medskip


\section{Introduction}

Let \(p(n)\) denote the number of unrestricted partitions of the
positive integer \(n\). Thus
\[
p(1)=1,\quad
p(2)=2,\quad
p(3)=3,\quad
p(4)=5,\quad
p(5)=7,\quad
p(6)=11,\ldots
\]
The partition function is one of the classical objects of additive
number theory and combinatorics; see, for example,
Andrews~\cite{Andrews} and Andrews--Eriksson~\cite{AndrewsEriksson}.

In this paper we investigate the sequence
\[
H(n):=p(n)-n.
\]
Its first values are
\[
0,0,0,1,2,5,8,14,21,32,45,\ldots,
\]
and the sequence appears in the OEIS as A000094
\cite{OEIS}.

Our initial motivation for studying this sequence comes from Euler's
totient function. Recall that \(\varphi(n)\) denotes the number of
positive integers not exceeding \(n\) that are relatively prime to
\(n\). Its complementary quantity
\[
n-\varphi(n)
\]
therefore counts the positive integers not exceeding \(n\) that are
not relatively prime to \(n\); this quantity is commonly referred to
as the cototient (or co-totient) function.

Questions concerning the values, ranges, and distribution of Euler's
totient function and related arithmetic functions have been studied
extensively. In particular, Luca and Pollack~\cite{LucaPollack}
consider questions concerning the values of Euler's totient function
and its range. Related questions involving the cototient function
and its interaction with other arithmetic functions have also been
studied by Luca and Pomerance~\cite{LucaPomeranceCototient}.

This arithmetic viewpoint provided the original motivation for the
present investigation. Just as \(\varphi(n)\) is naturally compared
with the ambient quantity \(n\), it is natural to ask what can be
learned by comparing the classical partition-counting function
\(p(n)\) with \(n\). This leads to the sequence \(H(n)=p(n)-n\)
defined above. Our principal interest is not merely in the size of
\(H(n)\), but in the arithmetic structure of its range: which
positive integers are attained, which are omitted, how large the
gaps between consecutive attained values become, and how frequently
attained values occur.

In the partition setting, the difference \(p(n)-n\) has the
additional advantage of admitting an immediate combinatorial
interpretation. A hook-shaped partition of \(n\) is a partition of
the form
\[
(k,1^{\,n-k}),\qquad 1\le k\le n.
\]
Equivalently, its Ferrers diagram has Durfee square of side one.
Hook partitions occur throughout the theory of Young diagrams and
standard Young tableaux; see, for example,
Pak~\cite{PakSYT}. Related uses of hook-shaped partitions also occur
in representation-theoretic settings; see Heim and
Luo~\cite{HeimLuo}.

There are exactly \(n\) hook partitions of \(n\), and consequently
\[
H(n)=p(n)-n
\]
counts precisely the non-hook partitions of \(n\), or equivalently
the partitions whose Durfee square has side at least two. Thus the
arithmetic motivation above leads to a sequence that also has a
natural partition-theoretic interpretation.

The same sequence occurs in the enumeration of unlabeled trees of
diameter four. The enumeration of trees according to height and
diameter goes back at least to Riordan~\cite{Riordan}; see also
Knopfmacher, Tichy, Wagner and Ziegler~\cite{KTWZ} for related
connections between graphs and partitions. Hence \(H(n)\) arises
naturally from several different points of view.

The classical asymptotic formula of Hardy and Ramanujan
\cite{HardyRamanujan} states that
\[
p(n)\sim
\frac{1}{4\sqrt3\,n}
\exp\left(\pi\sqrt{\frac{2n}{3}}\right).
\]
Rademacher~\cite{Rademacher} subsequently obtained an exact
convergent series for \(p(n)\). These analytic results describe the
size of \(p(n)\), and therefore of \(H(n)\), very accurately.

Our purpose here is somewhat different. We study the \emph{range}
\[
\cH^+:=\{H(n):n\ge4\}
\]
as a subset of the positive integers. More specifically, we ask:

\begin{itemize}[leftmargin=2em]
\item Which positive integers are attained by \(H(n)\)?
\item How large are the blocks of integers omitted between
      consecutive attained values?
\item How many attained values lie below a given real number \(x\)?
\item How sparse is the range inside the positive integers?
\end{itemize}

We first prove that \(H(n)\) is strictly increasing from \(n=3\)
onward. This allows the omitted positive integers to be decomposed
exactly into disjoint blocks. If
\[
G(n)=H(n+1)-H(n)-1,
\]
denotes the number of omitted integers between two consecutive
attained values, then
\[
G(n)=p(n+1)-p(n)-2.
\]
Thus the local geometry of the range is controlled by the first
difference of the partition function. Finite differences of
\(p(n)\) have been studied more generally by Gupta~\cite{Gupta}.

Using asymptotic information for shifted quotients of the partition
function, we obtain
\[
G(n)\sim
\frac{\pi}{\sqrt{6n}}p(n).
\]
Hence the omitted blocks grow without bound in absolute size, even
though their size relative to \(H(n)\) tends to zero.

Our main global object is the range-counting function
\[
I(x)
=
\#\{m\in\cH^+:m\le x\}.
\]
The central analytic problem is therefore an asymptotic inversion
problem. If
\[
H(N)\le x<H(N+1),
\]
then \(I(x)=N-3\), and one must invert the exponential growth of
\(H(N)\) in order to recover \(N\) from \(x\).

We prove the three-term asymptotic expansion
\[
\begin{aligned}
I(x)
={}&
\frac{3}{2\pi^2}(\log x)^2
+\frac{6}{\pi^2}(\log x)(\log\log x)\\
&+
\frac{3}{\pi^2}
\log\left(\frac{6\sqrt3}{\pi^2}\right)\log x
+O\bigl((\log\log x)^2\bigr).
\end{aligned}
\]
In particular,
\[
I(x)\sim\frac{3}{2\pi^2}(\log x)^2,
\]
and consequently the positive range of \(H\) has natural density
zero.

Finally, we consider a natural generalization involving Durfee
squares. For a fixed integer \(r\ge2\), let \(H_r(n)\) denote the
number of partitions of \(n\) whose Durfee square has side at least
\(r\). Then \(H_2(n)=H(n)\), and we show that, for every fixed \(r\),
the corresponding positive range has the same leading-order
sparsity:
\[
I_r(x)\sim\frac{3}{2\pi^2}(\log x)^2.
\]

\section{Combinatorial interpretations and generating functions}

We begin with the elementary combinatorial interpretation of
\(H(n)\).

\begin{proposition}
For every \(n\ge1\), \(H(n)=p(n)-n\) is the number of non-hook
partitions of \(n\).  Equivalently,
\[
H(n)
=
\#\{\lambda\vdash n:\Dur(\lambda)\ge2\}.
\]
\end{proposition}

\begin{proof}
A hook partition of \(n\) has the form
\[
(k,1^{\,n-k}),\qquad 1\le k\le n.
\]
There are therefore exactly \(n\) hook partitions of \(n\).
Subtracting them from the \(p(n)\) unrestricted partitions gives
\[
H(n)=p(n)-n.
\]
A Ferrers diagram is a hook precisely when its Durfee square has
side one, proving the equivalent formulation.
\end{proof}

\begin{figure}[ht]
\centering
\begin{tikzpicture}[scale=0.42]

\foreach \x in {0,...,6}
  \draw (\x,4) rectangle ++(1,1);
\foreach \y in {0,...,3}
  \draw (0,\y) rectangle ++(1,1);

\node at (3.5,-1.0) {\small A hook partition};
\node at (3.5,-1.8) {\small \(\Dur(\lambda)=1\)};

\begin{scope}[xshift=12cm]
  \foreach \x in {0,...,2}
    \foreach \y in {2,...,4}
      \draw (\x,\y) rectangle ++(1,1);

  \foreach \x in {3,...,5}
    \draw (\x,4) rectangle ++(1,1);
  \foreach \x in {3,...,4}
    \draw (\x,3) rectangle ++(1,1);
  \draw (3,2) rectangle ++(1,1);

  \foreach \x in {0,...,2}
    \draw (\x,1) rectangle ++(1,1);
  \foreach \x in {0,...,1}
    \draw (\x,0) rectangle ++(1,1);

  \node at (3,-1.0) {\small Durfee-square decomposition};
  \node at (3,-1.8) {\small \(d=3\)};
\end{scope}

\end{tikzpicture}
\caption{A hook partition and a typical Durfee-square decomposition.}
\label{fig:durfee}
\end{figure}
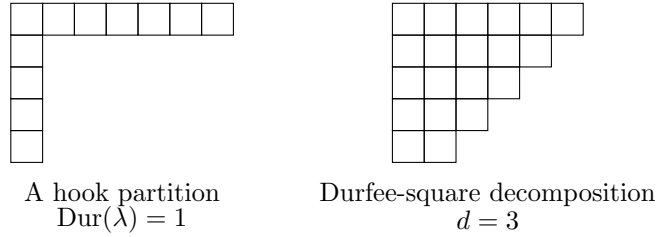

The ordinary partition generating function is
\[
\sum_{n\ge0}p(n)q^n
=
\frac{1}{(q;q)_\infty},
\]
where
\[
(q;q)_\infty=\prod_{j\ge1}(1-q^j).
\]

Since
\[
\sum_{n\ge1}nq^n=\frac{q}{(1-q)^2},
\]
we obtain the following generating function.

\begin{proposition}
We have
\[
\sum_{n\ge1}H(n)q^n
=
\frac{1}{(q;q)_\infty}
-1-\frac{q}{(1-q)^2}.
\]
\end{proposition}

\begin{proof}
Using \(H(n)=p(n)-n\),
\[
\begin{aligned}
\sum_{n\ge1}H(n)q^n
&=
\sum_{n\ge1}p(n)q^n-\sum_{n\ge1}nq^n\\
&=
\left(\frac1{(q;q)_\infty}-1\right)
-\frac{q}{(1-q)^2}.
\end{aligned}
\]
\end{proof}

The interpretation in terms of non-hook partitions also explains
the appearance of this sequence in certain enumerations of
unlabeled trees of diameter four; see \cite{KTWZ,Riordan}.  We will
not need the tree interpretation in the proofs below, but it
provides an additional combinatorial realization of the same
sequence.

\section{Monotonicity and omitted blocks}

Our first goal is to understand the ordering of the attained values.

\begin{lemma}
For \(n\ge1\),
\[
p(n)-p(n-1)
\]
equals the number of partitions of \(n\) containing no part equal to
\(1\).
\end{lemma}

\begin{proof}
Removing one part equal to \(1\) gives a bijection between
partitions of \(n\) containing a \(1\) and partitions of \(n-1\).
Thus exactly \(p(n-1)\) of the \(p(n)\) partitions of \(n\) contain
a \(1\), and the result follows.
\end{proof}

\begin{theorem}
For every \(n\ge3\),
\[
H(n+1)>H(n).
\]
\end{theorem}

\begin{proof}
We have
\[
H(n+1)-H(n)
=
p(n+1)-p(n)-1.
\]
By the preceding lemma, \(p(n+1)-p(n)\) counts partitions of \(n+1\)
with no part equal to \(1\).  For \(n\ge3\), the two distinct
partitions
\[
(n+1)
\qquad\text{and}\qquad
(n-1,2)
\]
are both of this type.  Hence
\[
p(n+1)-p(n)\ge2,
\]
and therefore
\[
H(n+1)-H(n)\ge1.
\]
\end{proof}

The starting point is sharp, since
\[
H(1)=H(2)=H(3)=0.
\]

For \(n\ge3\), define the \emph{gap length}
\[
G(n):=H(n+1)-H(n)-1.
\]
Thus \(G(n)\) is the number of omitted integers strictly between the
consecutive attained values \(H(n)\) and \(H(n+1)\).  Since
\(H(n)=p(n)-n\),
\[
\boxed{
G(n)=p(n+1)-p(n)-2.
}
\]

For example,
\[
H(10)=32,\qquad H(11)=45,
\]
and hence the omitted block between them is
\[
33,34,\ldots,44,
\]
of length
\[
G(10)=12.
\]

Define
\[
M(N)
=
\#\{m\in\Z_{>0}:m\le H(N),\ m\notin\cH^+\}.
\]

\begin{theorem}
For \(N\ge4\),
\[
M(N)=p(N)-2N+3.
\]
\end{theorem}

\begin{proof}
By strict monotonicity, the positive attained values not exceeding
\(H(N)\) are precisely
\[
H(4),H(5),\ldots,H(N),
\]
of which there are \(N-3\).  Among the \(H(N)\) positive integers
not exceeding \(H(N)\), the remaining integers are omitted.
Therefore
\[
M(N)
=
H(N)-(N-3)
=
p(N)-2N+3.
\]
\end{proof}

\begin{corollary}
For \(N\ge4\),
\[
\sum_{n=3}^{N-1}G(n)
=
p(N)-2N+3.
\]
Equivalently,
\[
M(N)=\sum_{n=3}^{N-1}G(n).
\]
\end{corollary}

\begin{proof}
Using
\[
G(n)=H(n+1)-H(n)-1,
\]
we obtain by telescoping
\[
\begin{aligned}
\sum_{n=3}^{N-1}G(n)
&=
H(N)-H(3)-(N-3)\\
&=
H(N)-(N-3)\\
&=
p(N)-2N+3.
\end{aligned}
\]
\end{proof}

We also obtain an exact decomposition of all omitted positive
integers.

\begin{theorem}
The complement of the positive range is
\[
\Z_{>0}\setminus\cH^+
=
\bigcup_{n\ge3}
\{H(n)+1,\ldots,H(n+1)-1\},
\]
where an interval is understood to be empty when
\(H(n+1)=H(n)+1\).  The intervals in this union are pairwise
disjoint.
\end{theorem}

\begin{proof}
The sequence
\[
H(3),H(4),H(5),\ldots
\]
is strictly increasing and integer-valued.  Hence
\[
H(n+1)\ge H(n)+1,
\]
and in particular \(H(n)\to\infty\).  Every positive integer not
attained by \(H\) therefore lies strictly between a unique pair of
consecutive attained values.  This gives the stated decomposition.
\end{proof}

\section{Asymptotic size of the omitted blocks}

Set
\[
B:=\pi\sqrt{\frac23}.
\]
The Hardy--Ramanujan asymptotic gives
\[
p(n)
=
\frac{e^{B\sqrt n}}{4\sqrt3\,n}
\left(1+O(n^{-1/2})\right).
\]

For the first difference, however, the leading equivalence alone is
not sufficient to justify taking the difference of two neighboring
values.  We therefore use a shifted-quotient expansion directly.
The asymptotic theory for shifted quotients of the partition
function gives, for fixed shift \(1\),
\cite{BanerjeePauleRaduSchneider}
\[
\frac{p(n+1)}{p(n)}
=
1+\frac{\pi}{\sqrt{6n}}+O(n^{-1}).
\]
This is also consistent with the classical
Hardy--Ramanujan--Rademacher theory and with the literature on
finite differences of \(p(n)\); see \cite{Gupta,Rademacher}.

\begin{theorem}
As \(n\to\infty\),
\[
G(n)\sim
\frac{\pi}{\sqrt{6n}}p(n).
\]
Equivalently,
\[
G(n)
\sim
\frac{\pi}{12\sqrt2\,n^{3/2}}
\exp\left(\pi\sqrt{\frac{2n}{3}}\right).
\]
Moreover,
\[
\frac{G(n)}{H(n)}
\sim
\frac{\pi}{\sqrt{6n}}.
\]
\end{theorem}

\begin{proof}
From the shifted-quotient expansion,
\[
\begin{aligned}
p(n+1)-p(n)
&=
p(n)
\left(
\frac{\pi}{\sqrt{6n}}+O(n^{-1})
\right)\\
&=
\frac{\pi}{\sqrt{6n}}p(n)
\left(1+O(n^{-1/2})\right).
\end{aligned}
\]
Since
\[
\frac{p(n)}{\sqrt n}\to\infty,
\]
we have
\[
2=o\left(\frac{p(n)}{\sqrt n}\right).
\]
Using
\[
G(n)=p(n+1)-p(n)-2,
\]
we therefore obtain
\[
G(n)\sim
\frac{\pi}{\sqrt{6n}}p(n).
\]

Substitution of
\[
p(n)\sim
\frac{1}{4\sqrt3\,n}
e^{B\sqrt n}
\]
gives
\[
G(n)
\sim
\frac{\pi}{12\sqrt2\,n^{3/2}}
e^{B\sqrt n},
\]
which is the second assertion.

Finally,
\[
H(n)=p(n)-n\sim p(n),
\]
so
\[
\frac{G(n)}{H(n)}
\sim
\frac{\pi}{\sqrt{6n}}.
\]
\end{proof}

\begin{corollary}
The omitted blocks have unbounded absolute length:
\[
G(n)\to\infty.
\]
At the same time,
\[
\frac{G(n)}{H(n)}\to0.
\]
Consequently,
\[
\frac{H(n+1)}{H(n)}\to1.
\]
\end{corollary}

Thus consecutive elements of the range become increasingly far
apart in absolute terms, while becoming increasingly close relative
to their magnitude.

\section{Counting the attained and omitted values}

Define
\[
I(x)
=
\#\{m\in\cH^+:m\le x\}.
\]
If \(N=N(x)\) is determined by
\[
H(N)\le x<H(N+1),
\]
then
\[
I(x)=N-3.
\]

We shall invert the asymptotic growth of \(H(N)\).

From the Hardy--Ramanujan--Rademacher theory,
\[
p(n)
=
\frac{e^{B\sqrt n}}{4\sqrt3\,n}
\left(1+O(n^{-1/2})\right).
\]
Since
\[
\frac{n}{p(n)}
\]
is exponentially small in \(\sqrt n\), subtraction of \(n\) does
not affect this relative error.  Hence
\[
H(n)
=
\frac{e^{B\sqrt n}}{4\sqrt3\,n}
\left(1+O(n^{-1/2})\right).
\]
Taking logarithms,
\[
\boxed{
\log H(n)
=
B\sqrt n-\log n-\log(4\sqrt3)
+O(n^{-1/2}).
}
\]

\begin{lemma}
As \(n\to\infty\),
\[
\log H(n+1)-\log H(n)
=
O(n^{-1/2}).
\]
\end{lemma}

\begin{proof}
Subtracting the preceding logarithmic expansions gives
\[
\begin{aligned}
\log H(n+1)-\log H(n)
={}&
B(\sqrt{n+1}-\sqrt n)\\
&-\log\left(1+\frac1n\right)
+O(n^{-1/2}).
\end{aligned}
\]
Now
\[
\sqrt{n+1}-\sqrt n=O(n^{-1/2})
\]
and
\[
\log\left(1+\frac1n\right)=O(n^{-1}),
\]
which proves the claim.
\end{proof}

Consequently, if
\[
H(N)\le x<H(N+1),
\]
then
\[
\boxed{
\log x=\log H(N)+O(N^{-1/2}).
}
\]

We now obtain the main inverse asymptotic.

\begin{theorem}
As \(x\to\infty\),
\[
\begin{aligned}
I(x)
={}&
\frac{3}{2\pi^2}(\log x)^2\\
&+
\frac6{\pi^2}(\log x)(\log\log x)\\
&+
\frac3{\pi^2}
\log\left(\frac{6\sqrt3}{\pi^2}\right)\log x\\
&+
O\bigl((\log\log x)^2\bigr).
\end{aligned}
\]
In particular,
\[
I(x)\sim
\frac{3}{2\pi^2}(\log x)^2.
\]
\end{theorem}

\begin{proof}
Choose \(N=N(x)\) so that
\[
H(N)\le x<H(N+1).
\]
Then
\[
I(x)=N-3.
\]

Set
\[
L=\log x,\qquad
y=\sqrt N,\qquad
c=\log(4\sqrt3).
\]
By the preceding lemma and the logarithmic expansion for \(H(N)\),
\[
L
=
By-2\log y-c+O(y^{-1}).
\]
In particular,
\[
y\asymp L.
\]

Define
\[
C:=c-2\log B.
\]
Since
\[
B^2=\frac{2\pi^2}{3},
\]
we have
\[
\boxed{
C
=
\log\left(\frac{6\sqrt3}{\pi^2}\right).
}
\]

Consider the approximation
\[
y_0
=
\frac1B(L+2\log L+C).
\]
Factoring out \(L\) gives
\[
y_0
=
\frac LB
\left(
1+\frac{2\log L+C}{L}
\right),
\]
and hence
\[
\log y_0
=
\log L-\log B
+
O\left(\frac{\log L}{L}\right).
\]

Let
\[
f(t)=Bt-2\log t-c.
\]
The true solution satisfies
\[
f(y)=L+O(L^{-1}),
\]
because \(y\asymp L\).  On the other hand, the definition of \(y_0\)
and the preceding estimate give
\[
f(y_0)
=
L+O\left(\frac{\log L}{L}\right).
\]
Therefore
\[
f(y)-f(y_0)
=
O\left(\frac{\log L}{L}\right).
\]

For \(t\asymp L\),
\[
f'(t)=B-\frac2t=B+O(L^{-1}),
\]
so \(f'(t)\) is bounded away from zero for sufficiently large \(L\).
The mean value theorem therefore yields
\[
y-y_0
=
O\left(\frac{\log L}{L}\right).
\]
Thus
\[
\boxed{
y
=
\frac1B(L+2\log L+C)
+
O\left(\frac{\log L}{L}\right).
}
\]

Squaring,
\[
N
=
\frac1{B^2}(L+2\log L+C)^2
+
O(\log L).
\]
Expanding the square and absorbing the lower-order terms gives
\[
N
=
\frac1{B^2}L^2
+
\frac4{B^2}L\log L
+
\frac{2C}{B^2}L
+
O((\log L)^2).
\]
Since
\[
\frac1{B^2}
=
\frac3{2\pi^2},
\]
we obtain
\[
N
=
\frac3{2\pi^2}L^2
+
\frac6{\pi^2}L\log L
+
\frac3{\pi^2}CL
+
O((\log L)^2).
\]

Finally,
\[
L=\log x,\qquad
\log L=\log\log x,
\]
and
\[
C=
\log\left(\frac{6\sqrt3}{\pi^2}\right).
\]
Since \(I(x)=N-3\), the constant \(-3\) is absorbed into the error
term, giving
\[
\begin{aligned}
I(x)
={}&
\frac{3}{2\pi^2}(\log x)^2
+\frac6{\pi^2}(\log x)(\log\log x)\\
&+
\frac3{\pi^2}
\log\left(\frac{6\sqrt3}{\pi^2}\right)\log x
+
O((\log\log x)^2).
\end{aligned}
\]
\end{proof}

\begin{corollary}
The positive range \(\cH^+\) has natural density zero:
\[
\lim_{x\to\infty}\frac{I(x)}x=0.
\]
\end{corollary}

\begin{proof}
The theorem gives
\[
I(x)=O((\log x)^2),
\]
and
\[
\frac{(\log x)^2}{x}\to0.
\]
\end{proof}

To avoid confusion with Big-\(O\) notation, define the complementary
counting function by
\[
C(x)
=
\#\{m\in\Z_{>0}:m\le x,\ m\notin\cH^+\}.
\]

\begin{corollary}
As \(x\to\infty\),
\[
\begin{aligned}
C(x)
={}&
x-\frac3{2\pi^2}(\log x)^2\\
&-\frac6{\pi^2}(\log x)(\log\log x)\\
&-\frac3{\pi^2}
\log\left(\frac{6\sqrt3}{\pi^2}\right)\log x\\
&+
O((\log\log x)^2).
\end{aligned}
\]
In particular,
\[
C(x)\sim x,
\]
so the omitted positive integers have natural density one.
\end{corollary}

\begin{proof}
Exactly,
\[
C(x)=\lfloor x\rfloor-I(x).
\]
Since
\[
\lfloor x\rfloor=x+O(1),
\]
the result follows from the preceding theorem.

At an attained value \(x=H(N)\), this is consistent with the exact
formula of Section~3:
\[
I(H(N))=N-3
\]
and
\[
\begin{aligned}
C(H(N))
&=
H(N)-N+3\\
&=
p(N)-2N+3.
\end{aligned}
\]
\end{proof}

\section{A fixed Durfee-square threshold}

We now consider a natural generalization.  For a fixed integer
\(r\ge2\), define
\[
H_r(n)
=
\#\{\lambda\vdash n:\Dur(\lambda)\ge r\}.
\]
Thus
\[
H_2(n)=H(n).
\]

The classical Durfee-square decomposition gives
\[
\frac1{(q;q)_\infty}
=
\sum_{d\ge0}
\frac{q^{d^2}}{(q;q)_d^2}.
\]
Indeed, a partition whose Durfee square has side exactly \(d\)
consists of a \(d\times d\) square, a partition to its right with at
most \(d\) parts, and a partition below it whose largest part is at
most \(d\).  Hence its generating function is
\[
\frac{q^{d^2}}{(q;q)_d^2}.
\]

It follows that
\[
\boxed{
\sum_{n\ge0}H_r(n)q^n
=
\sum_{d\ge r}
\frac{q^{d^2}}{(q;q)_d^2}.
}
\]
Equivalently,
\[
\sum_{n\ge0}H_r(n)q^n
=
\frac1{(q;q)_\infty}
-
\sum_{d=0}^{r-1}
\frac{q^{d^2}}{(q;q)_d^2}.
\]

\begin{proposition}
For fixed \(r\ge2\),
\[
H_r(n)=0\qquad(0\le n<r^2),
\]
\[
H_r(r^2)=1,
\]
and
\[
H_r(n)>H_r(n-1)
\qquad(n\ge r^2).
\]
\end{proposition}

\begin{proof}
A partition whose Durfee square has side at least \(r\) contains at
least \(r^2\) boxes.  Hence
\[
H_r(n)=0
\qquad(n<r^2).
\]
At \(n=r^2\), the unique such partition is
\[
(r^r)=(\underbrace{r,r,\ldots,r}_{r\text{ times}}),
\]
so
\[
H_r(r^2)=1.
\]

For the strict increase, removing one part equal to \(1\) gives a
bijection between partitions counted by \(H_r(n)\) that contain a
\(1\) and partitions counted by \(H_r(n-1)\).  Consequently,
\[
H_r(n)-H_r(n-1)
\]
counts partitions \(\lambda\vdash n\) satisfying
\[
\Dur(\lambda)\ge r
\]
and containing no part equal to \(1\).

Write
\[
n=r^2+s,\qquad s\ge0.
\]
The partition
\[
(r+s,\underbrace{r,\ldots,r}_{r-1\text{ times}})
\]
has size \(n\), has Durfee square of side at least \(r\), and
contains no part equal to \(1\).  Therefore
\[
H_r(n)-H_r(n-1)\ge1,
\]
which proves strict monotonicity.
\end{proof}

For clarity, let
\[
D_d(n)
=
\#\{\lambda\vdash n:\Dur(\lambda)=d\}.
\]
Then
\[
\sum_{n\ge0}D_d(n)q^n
=
\frac{q^{d^2}}{(q;q)_d^2}.
\]
For fixed \(d\), this is a rational function whose poles are roots
of unity.  Consequently its coefficients grow at most polynomially:
there exists an exponent \(A_d\) such that
\[
D_d(n)=O_d(n^{A_d}).
\]

Since \(r\) is fixed,
\[
\begin{aligned}
H_r(n)
&=
p(n)-\sum_{d=0}^{r-1}D_d(n)\\
&=
p(n)-O_r(n^{A_r})
\end{aligned}
\]
for a suitable exponent \(A_r\).  The Hardy--Ramanujan asymptotic
then gives
\[
\boxed{
H_r(n)\sim p(n)
}
\qquad(n\to\infty)
\]
for every fixed \(r\ge2\).

Define the positive range
\[
\cH_r^+
=
\{H_r(n):n\ge r^2\}
\]
and its counting function
\[
I_r(x)
=
\#\{m\in\cH_r^+:m\le x\}.
\]

By strict monotonicity,
\[
H_r(N)\le x<H_r(N+1)
\quad\Longrightarrow\quad
\boxed{
I_r(x)=N-r^2+1.
}
\]

\begin{theorem}
For every fixed \(r\ge2\),
\[
I_r(x)
\sim
\frac{3}{2\pi^2}(\log x)^2
\qquad(x\to\infty).
\]
Consequently, \(\cH_r^+\) has natural density zero.
\end{theorem}

\begin{proof}
For fixed \(r\),
\[
H_r(n)=p(n)(1+o(1)),
\]
and the excluded contribution is exponentially small in \(\sqrt n\)
relative to \(p(n)\).  Hence
\[
\log H_r(n)
=
B\sqrt n-\log n+O_r(1).
\]
If
\[
H_r(N)\le x<H_r(N+1),
\]
the leading inversion therefore gives
\[
\log x\sim B\sqrt N.
\]
Thus
\[
N
\sim
\frac{1}{B^2}(\log x)^2
=
\frac{3}{2\pi^2}(\log x)^2.
\]
Since
\[
I_r(x)=N-r^2+1
\]
and \(r\) is fixed, the asserted asymptotic follows.  Dividing by
\(x\) then gives density zero.
\end{proof}

\begin{remark}
Because, for fixed \(r\), the excluded Durfee classes are
exponentially small in \(\sqrt n\) relative to \(p(n)\), the
asymptotic inversion above can be refined further.  In particular,
one may obtain lower-order terms analogous to those in the
three-term expansion for \(I(x)\).  We do not pursue these
refinements here.
\end{remark}

\section{Concluding remarks}

We have studied the range of the partition-derived sequence
\[
H(n)=p(n)-n
\]
from both exact and asymptotic points of view.

The elementary identity
\[
G(n)=p(n+1)-p(n)-2
\]
connects the omitted blocks in the range directly with the first
difference of the partition function.  Analytically,
\[
G(n)
\sim
\frac{\pi}{\sqrt{6n}}p(n),
\]
so the omitted blocks become arbitrarily long, whereas
\[
\frac{G(n)}{H(n)}\to0.
\]
Thus the range exhibits the somewhat contrasting phenomena of
increasing absolute separation and decreasing relative separation.

Globally, the range-counting function satisfies
\[
\begin{aligned}
I(x)
={}&
\frac{3}{2\pi^2}(\log x)^2
+\frac6{\pi^2}(\log x)(\log\log x)\\
&+
\frac3{\pi^2}
\log\left(\frac{6\sqrt3}{\pi^2}\right)\log x
+
O((\log\log x)^2).
\end{aligned}
\]
This quantifies the extreme sparsity of the attained values and
shows that their natural density is zero.

More generally, for every fixed \(r\ge2\), the range of the
Durfee-threshold counting function \(H_r(n)\) has the same
leading-order sparsity:
\[
I_r(x)
\sim
\frac{3}{2\pi^2}(\log x)^2.
\]

Several questions remain natural.  One may seek further terms in
the inverse expansion for \(I(x)\), investigate congruence
properties of the attained and omitted values, or study the
Durfee-threshold problem when \(r=r(n)\) is allowed to grow with
\(n\).  In the latter setting the fixed-\(r\) polynomial-error
argument no longer applies directly, and the distribution of
Durfee-square sizes becomes relevant.

More broadly, it would be interesting to investigate analogous
range and omitted-value problems for other partition-derived
counting sequences.

\section{Acknowledgements}

The second author carried out this work during her summer internship at SRM University-AP from June 16 to July 24, 2026. The authors gratefully acknowledge SRM University-AP for providing the research facilities and a supportive research environment that facilitated this work.


\end{document}